\documentclass[11pt,a4paper]{amsart}
\usepackage{graphicx,amssymb}
\input xy
\xyoption{all}
\usepackage[all]{xy}
\usepackage{hyperref}

\newtheorem{thm}{Theorem}[section]
\newtheorem{cor}[thm]{Corollary}
\newtheorem{lem}[thm]{Lemma}
\newtheorem{prop}[thm]{Proposition}
\theoremstyle{definition}

\newtheorem{rem}[thm]{Remark}

\numberwithin{equation}{section}

\newcommand{\QQ}{\mathbb Q}

\newcommand{\ZZ}{\mathbb Z}
\newcommand{\CC}{\mathbb C}
\newcommand{\PP}{\mathbb P}
\newcommand{\FF}{\mathbb F}

\newcommand{\lra}{\longrightarrow}
\newcommand{\hra}{\hookrightarrow}
\newcommand{\ra}{\rightarrow}

\newcommand{\cA}{\mathcal{A}}

\newcommand{\cK}{\mathcal{K}}

\newcommand{\tC}{\widetilde{C}}

\newcommand{\cC}{\mathcal{C}}
\newcommand{\cE}{\mathcal{E}}

\newcommand{\cM}{\mathcal{M}}
\newcommand{\cN}{\mathcal{N}}

\newcommand{\cO}{\mathcal{O}}
\newcommand{\cP}{\mathcal{P}}
\newcommand{\cR}{\mathcal{R}}
\newcommand{\cS}{\mathcal{S}}

\newcommand{\cZ}{\mathcal{Z}}
\newcommand{\cX}{\mathcal{X}}

\newcommand{\cB}{\mathcal{B}}

\DeclareMathOperator{\Aut}{{Aut}}

 \DeclareMathOperator{\Ker}{Ker}
  
\DeclareMathOperator{\Pic}{Pic}
 \DeclareMathOperator{\Nm}{{Nm}}
  \DeclareMathOperator{\Prym}{{Pr}}
 \DeclareMathOperator{\pr}{{Pr}}
 \DeclareMathOperator{\Ima}{{Im}}

\DeclareMathOperator{\Spec}{Spec}
 \DeclareMathOperator{\Stab}{Stab}

 \DeclareMathOperator{\im}{Im}
 
\DeclareMathOperator{\End}{{End}}

\begin{document}

\title[ ]{ The Prym map of degree-7 cyclic coverings}
\author{ Herbert Lange and  Angela Ortega}
\address{H. Lange \\ Department Mathematik der Universit\"at Erlangen \\ Germany}
\email{lange@math.fau.de}
              
\address{A. Ortega \\ Institut f\" ur Mathematik, Humboldt Universit\"at zu Berlin \\ Germany}
\email{ortega@math.hu-berlin.de}

\thanks{The second author was supported by Deutsche Forschungsgemeinschaft, SFB 647.}
\subjclass{14H40, 14H30}
\keywords{Prym variety, Prym map}%

\begin{abstract} 
We study the Prym map for degree-7 \'etale cyclic coverings over a curve of genus 2. We extend this map 
to a proper map on a partial compactification of the moduli space and prove that the Prym map is generically finite
onto its image of degree 10. 
\end{abstract}

\maketitle

\section{Introduction}

Consider an \'etale  finite covering $f: Y \ra X$ of degree $p$ of a smooth complex projective curve 
$X$ of genus $g \geq 2$.
Let  $\Nm_f : JY \ra JX$ denote the norm map of the corresponding Jacobians.
One can associate to the
covering $f$ its Prym variety 
$$
P(f):= (\Ker \Nm_f)^0,
$$ 
the connected component containing 0 of the kernel of the  
norm map, which is an abelian variety of dimension 
$$
\dim P(f) = g(Y) - g(X) = (p-1)(g-1).
$$ 
The variety $P(f)$ carries a 
natural  polarization namely,  the restriction of the principal polarization $\Theta_Y$ of $JY$ to $P(f)$.
Let $D$ denote the type of this polarization.
If moreover $f: Y \ra X$ is a cyclic covering of degree $p$, then the group action induces an action on
the Prym variety. Let $\cB_D$ denote the moduli space of
abelian varieties of dimension $(p-1)(g-1)$ with a polarization of  type $D$ and an 
automorphism of order $p$ compatible with the polarization. If $\cR_{g,p}$ denotes the moduli space of  \'etale cyclic coverings of 
degree $p$ of curves of genus $g$, we get a map
$$
\Prym_{g,p}: \cR_{g,p} \ra \cB_D
$$
associating to every covering in $\cR_{g,p}$ its Prym variety, called the {\it Prym map}. 

Particularly interesting are the cases where $\dim \cR_{g,p} = \dim \cB_D$. For instance, 
for  $p=2$  this occurs only if  $g=6$. In this case the Prym map  $\pr_{6,2} : \cR_6 \ra \cA_5$
is generically finite of degree 27 (see \cite{ds}) and the fibers carry the structure of the 27 lines on 
a smooth cubic surface.  For  $(g,p) = (4,3)$, it is also known that $\pr_{4,3}$ is generically finite of degree 16 
onto its 9-dimensional  image $\cB_D$ (see \cite{f}) .

In this paper we investigate the case $(g,p) = (2,7)$,  where $\dim \cR_{g,p} = \dim \cB_D$. 
The main result of the paper is the following theorem. Let $G$ be the cyclic group of order 7.

\begin{thm} \label{main-theorem}
For any \'etale $G$-cover $f: \widetilde C \ra C$ of a curve $C$ of genus $2$, 
the Prym variety $\Prym(f)$ is 
an abelian variety of dimension $6$ with a polarization of type $D=(1,1,1,1,1,7)$ and a $G$-action. The Prym map
$$
\Prym_{2,7}: \cR_{2,7} \ra \cB_D
$$
is generically finite of degree 10.
\end{thm}

The paper is organized as follows. First we compute in Section 2 the dimension of the moduli space $\cB_D$
when $(g,p)=(2,7)$. We also show that the case $(g,p)=(2,6)$,
mentioned in \cite{f} as one where the dimension of $\cR_{2,7}$ equals the dimension of the image of the Prym map,
does not have this property. In Sections 3-5,
we extend the Prym map to a partial compactification  of admissible coverings $\widetilde{\cR}_{2,7}$
such that $\Prym_{2,7}: \widetilde{\cR}_{2,7} \ra \cB_D$ is  a proper map. We prove the generic finiteness of the Prym map 
in Section 6 by specializing to a curve in the boundary. 
In order to compute the degree of the Prym map we describe in Section 7 a complete fiber over a special abelian sixfold
with polarization type $(1,1,1,1,1,7)$,  and in Section 8 we give a basis for the Prym differentials for the different types
of admissible coverings appearing in the special fiber. Finally, in Section 9 we determine the degree of the Prym map
by computing the local degrees along the special fiber.

We would like to thank E. Esteves for his useful suggestions for the proof of Theorem \ref{thm5.2}.
The second author is thankful to G. Farkas for stimulating discussions.

\section{Dimension of the moduli space $\cB_D$  }

As in the introduction, let $\cR_{2,7}$ denote the moduli space of non-trivial cyclic \'etale coverings 
$f: \tC \ra C$ of degree 7 of curves of genus 2. The Hurwitz formula gives $g(\tC) = 8$. 
Hence the Prym variety $P = P(f)$ is of dimension 6 and the canonical polarization of the Jacobian 
$J\tC$ induces a polarization of type $(1,1,1,1,1,7)$ on $P$. Let $\sigma$ denote an automorphism of 
$J\tC$ generating the group of automorphisms of $\tC/C$. It induces an automorphism of $P$,
also of order 7, which is compatible with the polarization.
The Prym map 
$
\Pr_{2,7} : \cR_{2,7} \ra \cB_D
$ 
is the morphism defined by $f \mapsto P(f)$. Here $\cB_D$ is the moduli space of abelian varieties of dimension 6 with a polarization 
of type $(1,1,1,1,1,7)$ and an automorphism of order 7 compatible with the polarization. 
The main result of this section is the following proposition.

\begin{prop} \label{p2.1}
$$
\dim \cB_D = \dim \cR_{2,7} =3.
$$
\end{prop}

\begin{proof}
Clearly $\dim \cR_{2,7} = \dim \cM_2 = 3$. So we have to show that also $\dim \cB_D = 3$.
For this we use Shimura's theory of abelian varieties with endomorphism structure (see \cite{sh} or
\cite[Chapter 9]{bl}).

Let $K = \QQ(\rho_7)$ denote the cyclotomic field generated by a primitive 7-th root of unity 
$\rho_7$. Clearly $\cB_D$ coincides 
with one of Shimura's moduli spaces of polarized abelian varieties with endomorphism structure in 
$K$. The field $K$ is a totally complex quadratic extension of a totally real number field of degree 
$e_0 = 3$. Denote 
$$
m := \frac{\dim P}{e_0} = 2.
$$
The polarization of $P$ depends on the lattice of $P$ and a skew-hermitian matrix $T \in 
M_m(\QQ(\rho_7))$. For each of the $e_0$  real embeddings of the totally real subfield of 
$\QQ(\rho_7)$ consider an extension $\QQ(\rho_7) \hra \CC$ and let $(r_{\nu},s_{\nu})$ be the 
signature of $T$ considered as a matrix in the extension $\CC$. The signature of $T$ is defined to 
be the $e_0$-tuple $((r_1,s_1).\dots, (r_{e_0},s_{s_0}))$
with 
$$
r_{\nu} + s_{\nu} = m = 2.
$$ 
for all $i$. Then, according to 
\cite[p. 162]{sh} or \cite[p. 266, lines 6-8]{bl} we have 
\begin{equation}\label{eq1.5}
\dim \cB_D = \sum_{\nu=1}^{e_0} r_{\nu} s_{\nu} \leq 3
\end{equation}
with equality if and only if  $r_{\nu}= s_{\nu} = 1$ for all $\nu$.

On the other hand, in Section 6 we will see that the map $\Pr_{2,7}$ is generically injective. This implies that 
$$
\dim \cB_D \geq \dim \cR_{2,7} = 3
$$
which completes the proof of the proposition.
\end{proof}

\begin{rem}
According to \cite{o} we know that $P$ is isogenous to the product of a Jacobian of dimension 3 
with itself. Then $\End_{\QQ}(P)$ is not a simple algebra. Hence, 
if one knew that  $\Pr_{2,7}$ is dominant onto the component $\cB_D$, then 
\cite[Proposition 9.9.1]{bl} implies that $r_{\nu} = s_{\nu} = 1$ for $\nu = 1,2,3$ which also gives 
$\dim \cB_D = 3$. 
\end{rem}

\begin{rem} In \cite{f} it is claimed that also the Prym map $\Pr_{2,6}: \cR_{2,6} \ra \cB_D$
satisfies $\dim \cB_D = \dim \cR_{2,6} = 3$.
However, we claim that the dimension of $\cB_D$ in this case cannot be $3$.

For the proof note that the cyclotomic field of the 6-th roots of unity $\QQ(\rho_6)$ is the imaginary quadratic field 
$\QQ(\sqrt{-3})$. So with the notation of the proof of Proposition \ref{p2.1} we have in this case
$$
e_0 = 1 \; \mbox{and} \; m = \frac{\dim P}{e_0} = 5
$$
and we have that
$$
\dim \cB_D = r_1s_1 \;\mbox{with}\; r_1 + s_1 = 5.
$$
So for $(r_1,s_1)$ there are the following possibilities (up to exchanging $r_1$ and $s_1$ which does not modify $\dim \cB_D$):
$$
(r_1,s_1) = (5,0), (4,1) \; \mbox{or} \;  (3,2)
$$
giving respectively
$$
\dim \cB_D  = 0, 4 \; \mbox{or} \; 6
$$
which in any case is different from 3.
\end{rem}

\section{The condition (*)}

In this section we study the Prym map for coverings of degree 7 between stable curves.
Let $G = \ZZ/7\ZZ$ be the the cyclic group of order 7 with generator $\sigma$ and $f:\widetilde C \ra 
C$ be a $G$-cover of a connected stable curve $C$ of arithmetic genus $g$. We fix in the sequel a
primitive 7-th root of the unity $\rho$. 
In this section we assume the following condition for the covering $f$.

$$
(*) \left\{ \begin{array}{l}
               \mbox{The fixed points of} \; \sigma \; \mbox{are exactly the nodes of} \; \widetilde C
\; \\
\mbox{and at each node 
               one local parameter is multiplied by}\\
                \rho^{\delta} \; \mbox{and the other
               by} \; \rho^{-\delta} \; \mbox{for some}  \; \delta, \; 1 \leq \delta \leq 3.
       \end{array}  \right.   
$$
As in \cite{b} we have $f^*\omega_{C} \simeq \omega_{\widetilde C}$ which implies
$$
p_a(\widetilde C) = 7g -6.
$$
Let $\widetilde N$ respectively $N$ be the normalization of $\widetilde C$, respectively $C$, and 
$\widetilde f: \widetilde N \ra N$ the induced map. At each node $s$ of $\widetilde C$ we make the usual  identification 
$$
\cK_s^*/\cO_s^* \simeq \CC^* \times \ZZ \times \ZZ.
$$
Then the action of $\sigma$ on $\cK_s^*/\cO_s^*$ is:
$$
\sigma^*((z,m,n)_s) = (\rho^{\delta(m+n)} z,m,n)
$$
for some $\delta, \;1 \leq \delta \leq 3$. Here we label the branches at the node $s$ such that a local 
parameter at 
the first branch (corresponding to $m$) is multiplied by $\rho^{\delta}$ with $1 \leq \delta \leq 3$. Then we have
$$
f_*((z,m,n)_s) = (\prod_{k=0}^6 (\sigma^{k})^*z,m,n)_{f(s)} = (z^7,m,n)_{f(s)}.
$$
We define the multidegree of a line bundle $L$ on $\widetilde C$ by
$$
\deg L = (d_1,\dots, d_v)
$$
where $v$ is the number of components of $\widetilde C$ and 
$d_i$ is the degree of $L$ on the $i$-th component of $\widetilde C$. 

\begin{lem} \label{lem2.1}
Let $L \in \Pic \widetilde C$ with $\Nm L \simeq \cO_{\widetilde C}$. Then 
$$
L \simeq M \otimes \sigma^*M^{-1}
$$
for some $M \in \Pic \widetilde C$. Moreover, $M$ can be chosen of multidegree $(k,0,\dots,0)$ with 
$0 \leq k \leq 6$.
\end{lem}

\begin{proof}
As in \cite[Lemma 1]{m} using Tsen's theorem, there is a divisor $D$ such that 
$L \simeq \cO_{\widetilde C}(D)$ and $f_*D = 0$. Writing $D = \sum_{x \in \widetilde C_{reg}} x + 
\sum_{s \in \widetilde C_{sing}} (z_s,m_s,n_s)$, we see that, at singular points 
$s \in \widetilde C$,  $D$ is a linear combination of divisors 
$x - \sigma^*x$ for $x \in \widetilde C_{reg}$ and $(\rho,0,0)_s$  (note that it suffices to show that $(\rho,0,0)$ is in the 
image of $1 - \sigma^*$ because $(\rho,0,0) + (\rho,0,0) = (\rho^2,0,0)$). If at $s$  $\delta$ is as above,
choose an integer $i$ such that  $-i\delta \equiv 1 \mod 7$. Then
$$
(1,i,0) - \sigma^*(1,i,0) = (\rho^{-\delta i},0,0) = (\rho,0,0).
$$
Hence $D = E - \sigma^*E$ for some divisor $E$ on $\widetilde C$.
Moreover, 
$$
(1,1,-1)_s - \sigma^*(1,1,-1)_s = (1,0,0)_s
$$ 
which altogether implies that 
$L \simeq M \otimes \sigma^*M^{-1}$,  where $M$ can be chosen of multidegree as stated.
\end{proof}

Let $P$ denote the Prym variety of $f: \widetilde C \ra C$, i.e. the connected component of 0
of the kernel of norm map $\Nm: J\widetilde C \ra JC$. By definition it is a connected commutative
algebraic group. Lemma \ref{lem2.1} implies that $P$ is the variety of line bundles in $\ker \Nm$ of the form 
$M \otimes \sigma^*M^{-1}$ with $M$ of multidegree $(0,\dots,0)$.

\begin{prop} \label{prop2.2}
Suppose $p_a(C) = g$. Then
$P$ is an abelian variety of dimension $6g-6$.
\end{prop}

\begin{proof} (As in \cite{b} and \cite{f}).
Consider the following diagram of commutative algebraic groups:
\begin{equation} \label{eq2.1}
\xymatrix{
	      0 \ar[r]  & \widetilde T \ar[r] \ar[d]^{\Nm} & J\widetilde C  \ar[r] \ar[d]^{\Nm} &J\widetilde N  \ar[r] \ar[d]^{\Nm} 
& 0\\
             0 \ar[r] &  T  \ar[r] &  JC  \ar[r] & JN \ar[r] & 0 
    }
\end{equation} 
where the vertical arrows are the norm maps and $T$ and $\widetilde T$ are the groups of  
classes of divisors of multidegree $(0,\dots,0)$ with singular support. Since $f^*$ is injective on $T$ 
and $\Nm \circ f^* = 7$, the norm on $\widetilde T$ is surjective and 
$$
\ker Nm_{|\widetilde T}  \simeq \widetilde T_7= \{\mbox{points of order $7$ in}\; \widetilde T\}.
$$
On the other hand, Lemma \ref{lem2.1} implies that 
$$
\ker (\Nm: J\widetilde C \ra JC) \simeq P \times \ZZ/7 \ZZ.
$$ 
Hence one obtains an exact sequence
\begin{equation} \label{eq2.2}
0 \ra \widetilde T_7 \ra P \times \ZZ/7\ZZ \ra R \ra 0.
\end{equation}
Suppose first that $C$ and hence $\widetilde C$ are nonsingular. Then $\widetilde C = \widetilde N$ 
and hence $\widetilde T =0$. Since $\ker (J\widetilde N \ra JN)$ has 7 components, $P$ is an abelian 
variety. 

Suppose that $C$ and thus also $\widetilde C$ have $s>0$ singular points. Then 
$\dim \widetilde T = \dim T = s$. Then $R$ is an abelian variety, since $\widetilde f$ is ramified. 
We get a surjective homomorphism $P \ra R$ with kernel consisting of $7^{s-1}$ elements. Hence 
also $P$ is an abelian variety. Moreover,
$$
\dim P = \dim R = \dim J\widetilde C - \dim JC. 
$$
Now, if $C$ has $s$ nodes and $N$ has $t$ connected components, then also $\widetilde C$ has $s$ 
nodes and $\widetilde N$ has $t$ connected components. This implies 
$$
\dim J\widetilde C - \dim JC= p_a(\widetilde C) - p_a(C) = 6g-6.
$$
\end{proof}

Let $\widetilde \Theta$ denote the canonical polarization of the generalized Jacobian $J\widetilde C$
(see \cite{b}). It restricts to a polarization $\Sigma$ on the abelian subvariety $P$. 
We denote the isogeny $P \ra \widehat P$ associated to $\Sigma$ by the same letter.

\begin{prop} \label{prop2.3}
The polarization $\Sigma$ on $P$ is of type $D:=(1,\dots,1,7,\dots,7)$ where $7$ occurs $g-1$ times 
and $1$ occurs $5g-5$ times.
\end{prop} 

\begin{proof}(Following \cite[Proposition 2.4]{f}). Consider the isogeny
$$
h: P \times JN \ra J \widetilde N.
$$
Clearly $\ker(h) \subset P[7]$. As in \cite[Proof of Theorem 3.7]{b}, one sees that $\ker(h)$ is 
isomorphic to
the group of points $a \in JC[7]$ such that $f^*a \in P$ and 
$$
\dim_{\FF_7} \ker(h) = \dim_{\FF_7} (JC[7]) -1 = 2g -t -1,
$$
where $t = \dim T = \dim \widetilde T$. Now $\ker(h)$ is a maximal isotropic subgroup of the kernel 
of the polarization of $P \times JN$, since this polarization is the pullback under $h$ of the principal 
polarization of $J \widetilde N$. This implies $\dim_{\FF_7} (\ker \Sigma \times JN[7]) = 4g-2t-2$.
Since $\dim_{\FF_7} (JN[7]) = 2g-2t$, it follows that $\dim_{\FF_7} (\ker \Sigma)= 2g-2$. Since 
$\ker \Sigma \subset P[7]$, this gives the assertion.
\end{proof}

\section{The condition (**)}

As in the last section let $f:\widetilde C \ra 
C$ be a $G$-covering of stable curves. Recall that  a node $z \in \widetilde C$ is either
\begin{itemize}
\item of index 1, i.e $|\Stab z| = 1$ in which case $f^{-1} (f(z))$ consists of 7 nodes 
which are cyclicly permuted under $\sigma$ or
\item of index 7, i.e. $|\Stab z| = 7$ in which case $z$ is the only preimage of the node $f(z)$ and $f$ 
is totally ramified at both branches of $z$. Since $\sigma$ is of order 7, the two branches of $z$
are not exchanged.
\end{itemize}

We also call a node of $C$ {\it of index} $i$ if a preimage (and hence every preimage) under $f$ is 
a node of index $i$.
We assume the following condition for the $G$-covering $f:\widetilde C \ra C$ of connected stable curves:

$$
(**) \left\{ \begin{array}{l}
                p_a(C) = g \; \mbox{and}  \; p_a(\widetilde C) = 7g-6;\\
                 \sigma \; \mbox{is not the identity on any irreducible component of} \;\widetilde C;\\
                \mbox{if at a fixed node of} \; \sigma \; \mbox{one local parameter is multiplied by} \;           
                 \rho^i, \; \mbox{the other is} \\
                 \mbox{multiplied by} \;  \rho^{-i}, \; \mbox{where} \; \rho \; \mbox{
denotes a fixed 7-th root of unity};\\
               P := \Prym(f) \; \mbox{is an abelian variety}. 
       \end{array}  \right.
$$
Under these assumptions the nodes of $\widetilde C$ are exactly the preimages of the nodes of $C$.
We denote for $i =1$ and 7:
\begin{itemize}
\item $n_i:=$ the number of nodes of $C$ of index $i$, i.e. nodes whose preimage consists of $\frac{7}{i}$ nodes of $\widetilde C$,
\item $c_i :=$ the number irreducible components of $C$ whose preimage consists of $\frac{7}{i}$ irreducible components of $\widetilde C$,
\item $r :=$ the number of fixed nonsingular points under $\sigma$.
\end{itemize}

\begin{lem} \label{lem3.1}
The covering satisfies $(**)$ if and only if $r=0$ and $c_1 = n_1$.
\end{lem}

In particular, any covering satisfying $(**)$ is an admissible $G$-cover (for the definition see 
Section 5 below).
                
\begin{proof} (As in \cite{b} and \cite{f}).
Let $\widetilde N$ respectively $N$ be the normalization of $\widetilde C$ respectively $C$. The covering $\widetilde f: \widetilde N \ra N$ 
is ramified exactly at the points lying over the fixed points of
$\sigma: \widetilde C \ra \widetilde C$. Hence the Hurwitz formula says
$$
p_a(\widetilde N) -1 = 7(p_a(N) -1) + 3r +6 n_7. 
$$
So 
\begin{eqnarray*} 
p_a(\widetilde C) -1 & = & p_a(\widetilde N) -1 + 7n_1 + n_7\\
& = & 7(p_a(N) -1) + 3r + 7n_1 + 7 n_7.
\end{eqnarray*}
Moreover,
$$
p_a(C) -1 = p_a(N) -1 +n_1 + n_7
$$ 
which altogether gives
$$
p_a(\widetilde C) -1 = 7(p_a(C) -1) + 3r.
$$ 
Hence the first condition in $(**)$ is equivalent to $r = 0$. 

Now we discuss the condition that $P$ is an 
abelian variety. For this consider again the diagram \eqref{eq2.1}.
From the surjectivity of the norm maps it follows that 
$P$ is an abelian variety if and only if $\dim \widetilde T = \dim T$. Now
$\dim J \widetilde N = p_a(\widetilde N) - n_7 - 7 n_1 +c_7 + 7c_1 -1$ and thus
$$
\dim \widetilde T = (n_7 - c_7) + 7(n_1 - c_1) + 1
$$ and 
$$
\dim T = (n_7 - c_7) + (n_1 -c_1) + 1.
$$
Hence $\dim \widetilde T = \dim T$ if and only if $c_1 = n_1$.
\end{proof}

Let $f: \widetilde C \ra C$ be a $G$-covering satisfying the condition $(**)$ with generating 
automorphism $\sigma$. We denote by $B$ the union of the components of $\widetilde C$ fixed under 
$\sigma$ and write 
$$
\widetilde C = A_1 \cup \cdots \cup A_7 \cup B
$$
with $\sigma(A_i) = A_{i+1}$ where $A_8 = A_1$.

\begin{prop} \label{prop3.2}
{\em (i)} If $B = \emptyset$, then $\widetilde C = A_1 \cup \cdots \cup A_7$ where $A_1$ 
can be chosen connected and tree-like and $\#A_i \cap A_{i+1} = 1$ for $i = 1, \dots, 7$.\\
{\em (ii)} If $B \neq \emptyset$, then $A_i \cap A_{i+1} = \emptyset$ for $i = 1, \dots , 7$.
Each connected component of $A_1$ is tree-like and meets $B$ at only one point. Also $B$ is connected.
\end{prop}

For the proof we need the following elementary lemma (the analogue of \cite[Lemma 5.3]{b} and 
\cite[Lemma 2.6]{f}) which will be applied to the dual graph of $\widetilde C$.

\begin{lem} \label{lem3.3}
Let $\Gamma$ be a connected graph with a fixed-point free automorphism $\sigma$ of order 7.
Then there exists a connected subgraph $S$ of $\gamma$ such that 
$\sigma^i(S) \cap \sigma^{i+1}(S) = \emptyset$ for $i = 0, \dots,6$ and $\cup_{i=0}^6 \sigma^i(S)$ contains every 
vertex of $\Gamma$. 
\end{lem}

\begin{proof}[Proof of Proposition \ref{prop3.2}] (As in \cite{b} and \cite{f}). Let $\Gamma$ denote 
the dual graph of $\widetilde C$.
If $B= \emptyset$, let $A_1$ correspond to the subgraph $S$ of Lemma \ref{lem3.3}. Let $v$ be
the number of vertices of $S$, $e$ the number of edges of $S$ and $s$ the numbers of nodes of 
$A_1$ which belong to only one component. The equality $c_1 = n_1$ implies 
$$
v = e + s - \#A_1 \cap A_2.
$$
Since $1 - v + e \geq 0$ and $\#A_1 \cap A_2 \geq 1$ gives $s=0, \; \#A_1 \cap A_2 = 1$ and 
$1-v+e = 0$. So $A_1$ is tree-like. This proves (i).

Asume $B \neq \emptyset$ and denote
\begin{itemize}
\item $t := \# A_1 \cap A_2, $
\item $m := \# A_i \cap B$ for $i = 1, \dots , 7$,
\item $i_{A_1} := \#$ irreducible components of $A_1$,
\item $c_{A_1} := \#$ of connected components of $A_1$,
\item $n_{A_1} := \#$ nodes of $A_1$.
\end{itemize}
Recall that assumption $(**)$ implies that $B$ does not contain any node which moves under 
$\sigma$. Then
$$
c_1 = i_{A_1}  \quad \mbox{and} \quad n_1 = n_{A_1} + r + m.
$$
For any curve we have $n_{A_1} - i_{A_1} + c_{A_1} \geq 0$ 
(see \cite[Proof of Lemma 5.3]{b}). Thus, if $c_1 = n_1$, 
$$
0 = n_{A_1} + t + m  - i_{A_1}\geq t+m - c_{A_1}.
$$
Since $\widetilde C$ is connected, any connected component of $\cup_{i=1}^7 A_i $ meets  $B$.
But then any connected component of $A_1$ meets $B$ which implies $m \geq c_{A_1}$. Hence
$$
0 \geq t+m-c_{A_1} \geq t \geq 0.
$$ 
Hence $t = 0, \; m= c_{A_1}$ and $n_{A_1} - i_{A_1} + c_{A_1} = 0$. So $A_i \cap A_{i+1} = 
\emptyset$ and $B$ is connected.
\end{proof}

\begin{thm}
Suppose that $f : \widetilde C \ra C$ satisfies condition $(**)$. Then there exist the following isomorphisms of polarized abelian varieties:
\\
In case {\em (i)}, \; $(P,\Sigma) \simeq \ker((JA_1)^7 \ra JA_1)$ with  the polarization induced by the principal polarization on $(JA_1)^7$.
\\
In case {\em (ii)}, \; $(P,\Sigma) \simeq  \ker((JA_1)^7 \ra JA_1) \times Q$, where $Q$ is the generalized Prym variety associated to the covering 
$B \ra B/ \sigma$.
\end{thm}

\begin{proof} (As in \cite[Theorem 5.4]{b}). In case (i), $\widetilde C$ is obtained from the disjoint 
union of 7 copies of $A_1$ by fixing 2 smooth points $p$ and $q$ of $A_1$  and identifying $q$ in the 
$i$-th copy  with $p$ in the $i+1$-th copy of $A_1$ cyclicly. The curve $C = \widetilde C/G$ is obtained from $A_1$ by 
identifying $p$ and $q$ and $f: \widetilde C \ra C$ is an \'etale covering. Note that $JA_1$ is an abelian 
variety, since $A_1$ is tree-like.

Consider the following diagram
\begin{equation*} 
\xymatrix{
0 \ar[r]  &  \CC^* \ar[r] \ar^{\Nm}[d] & J\widetilde C  \ar[r] \ar^{\Nm}[d] & (JA_1)^7  \ar[r] \ar^{m}[d] 
& 0\\
             0 \ar[r] &  \CC^*  \ar[r] &  JC  \ar[r] & JA_1 \ar[r] & 0 
    }
\end{equation*} 
where $m$ is the addition map. One checks immediately that $\Nm: \CC^* \ra \CC^*$  is an isomorphism. This implies the assertion.

In case (ii) we have
$$
J\widetilde C \simeq (JA_1)^7 \times JB \quad \mbox{and} \quad 
P = \ker(\Nm)^0 \simeq \ker((JA_1)^7 \ra JA_1) \times Q 
$$
which immediately implies the assertion.
\end{proof}


\section{The extension of the Prym map to a proper map}


Let $\cR_{g,7}$ denote the moduli space of non-trivial  \'etale $G$-covers $f: \widetilde C \ra C$ of smooth
 curves $C$ of genus $g$ and $\cB_D$ the moduli space of polarized abelian varieties of dimension 
$6g-6$ with polarization of type $D$ with $D$ as in Proposition \ref{prop2.3} and compatible with the $G$-action. 
As in the introduction we denote by
$$
\pr_{g,7}: \cR_{g,7} \ra \cB_D
$$ 
the corresponding Prym map associating to the covering $f$ the Prym variety $\Prym(f)$. 
In order to extend this map to a proper map we consider the compactification $\overline \cR_{g,7}$
of $\cR_{g,7}$ consisting of admissible $G$-coverings of stable curves of genus $g$ introduced in 
\cite{acv}.

Let $\cX \ra S$ be a family of stable curves of arithmetic genus $g$. 
A {\it family of admissible $G$-covers} of $\cX$ over $S$ is a finite morphism $\cZ \ra \cX$
such that,
\begin{enumerate}
\item the composition $\cZ \ra \cX \ra S$ is a family of stable curves;
\item every node of a fiber of $\cZ \ra S$ maps to a node of the corresponding fiber of $\cX \ra S$;
\item $\cZ \ra \cX$ is a principal $G$-bundle away from the nodes;
\item if $z$ is a node of index 7 in a fibre of $\cZ \ra S$ and $\xi$ and $\eta$ are local coordinates of 
the two branches near $z$, any element of 
the stabilizer 
$\Stab_G(z)$ acts as 
$$
(\xi, \eta) \mapsto (\rho \xi, \rho^{-1}\eta)
$$
where $\rho$ is a primitive $7$-th root of unity. 
\end{enumerate}

In the case of $S = \Spec \CC$ we just speak of an admissible $G$-cover.
In this case 
the ramification index at any node $z$ over $x$ equals the order of the stabilizer of $z$ and depends only on $x$. 
It is called the {\it index} of the $G$-cover $\cZ \ra \cX$ at $x$. Since 7 is a prime, the index of a node 
is either 1 or 7. 
Note that, for any admissible $G$-cover $Z \ra X$, the curve $Z$ is stable if and only if and only if $X$ is stable. 

As shown in \cite{acv} or \cite[Chapter 16]{acg}, the moduli space $\overline \cR_{g,7}$ of admissible 
$G$-covers stable of curves of genus $g$ is a natural compactification of $\cR_{g,7}$. 
Clearly the coverings satisfying condition $(**)$ are admissible and form an
open subspace $\widetilde \cR_{g,7}$ of $\overline \cR_{g,7}$.

\begin{thm} \label{thm4.1}
The map $\pr_{g,7}: \cR_{g,7} \ra \cB_D$ extends to a proper map 
$\widetilde {\pr}_{g,7}: \widetilde \cR_{g,7} \ra \cB_D$.
\end{thm}

\begin{proof} The proof is the same as the proof of  \cite[Theorem 2.8]{f} just replacing 
3-fold covers by 7-fold covers. So we will omit it. 
\end{proof}


\section{Generic finiteness of $\pr_{2,7}$}


From now on we consider only the case $g=2$, i.e. of $G$-covers of curves of genus 2. So 
$\dim \cR_{2,7} = \dim \cM_2 = 3$ and $\cB_D$ is the moduli space of polarized abelian varieties of 
type $(1,1,1,1,1,7)$ with $G$-action which is also of dimension 3.
Let $[f:\tC \ra C] \in \cR_{2,7}$ be a general point and let the covering $f$ be given by 
the $7$-division point $\eta \in JC$.

\begin{lem} \label{lem5.1}
{\em (i)} The cotangent space of $\cB_D$ at the point $\pr_{2,7}([f:\tC \ra C]) \in \cB_D$ is identified with the vector space
$\bigoplus_{i=1}^{3} \left( H^0(\omega_C \otimes \eta^{i}) \otimes H^0(\omega_C \otimes \eta^{7-i}) \right)$.

{\em (ii)} The codifferential of the map $\pr_{2,7}: \cR_{2,7} \ra \cB_D$ at the point $(f,\eta)$ is given by the sum of the multiplication maps
$$
\bigoplus_{i=1}^{3} \left( H^0(\omega_C \otimes \eta^{i}) \otimes H^0(\omega_C \otimes \eta^{7-i}) \right) \lra H^0(\omega_C^2).
$$
\end{lem} 

\begin{proof} 
(i): 
Consider the composed map 
$\cR_{g,7} \stackrel{\pr_{2,7}}{\lra} \cB_D \stackrel{\pi}{\lra} \cA_D$. The cotangent space of
the image of $[f:\tC \ra C]$ in $\cA_D$ is by definition the cotangent at the Prym variety $P$ of $f$. It is well known 
that the cotangent space $T^*_{P,0}$ at 0 is  
\begin{equation}  \label{eq5.1}
T^*_{P,0}  = H^0(\tC,\omega_{\tC})^- = \bigoplus_{i=1}^6 H^0(C, \omega_C \otimes \eta^i).
\end{equation}
According to \cite{lo} the cotangent space of $\cA_D$ at the point $P$ can be identified with the 
second symmetric product of $H^0(\tC,\omega_{\tC})^-$. This gives
\begin{equation}  \label{eq5.2}
T^*_{\cA_D,P} = \bigoplus_{i=1}^6 S^2 H^0(\omega_C \otimes \eta^i) \oplus \bigoplus_{i=1}^3
\left( H^0 (\omega_C \otimes \eta^i) \otimes H^0(\omega_C \otimes \eta^{7-i}) \right)
\end{equation}
Since the map $\pi: \cB_D \ra \cA_D$ is finite onto its image and the group $G$ acts on the cotangent space of $\cB_D$ at the point, we conclude that this space can be identified with a 
3-dimensional $G$-subspace of the $G$-space $T^*_{\cA_D,P}$ which is defined over the rationals. But there is only one such subspace, 
namely $\bigoplus_{i=1}^{3} \left( H^0(\omega_C \otimes \eta^{i}) \otimes H^0(\omega_C \otimes \eta^{7-i}) \right)$. This gives (i).

(ii): It is well known that the cotangent space of $\cR_{2,7}$ at a point $(C, \eta)$ without automorphism is given by $H^0(\omega_C^2)$ and the 
codifferential of $\pr_{2,7}: \cR_{2,7} \ra \cA_D$ at $(C,\eta)$ by the natural map $S^2(H^0(\tC, \omega_{\tC})^-) \lra H^0(\omega_C^2)$. 
The assertion follows immediately from Lemma \ref{lem5.1} (i) and equations \eqref{eq5.1} and \eqref{eq5.2}.
\end{proof}

\begin{thm} \label{thm5.2}
The map $\widetilde \pr_{2,7}: \widetilde \cR_{2,7} \ra \cB_D$ is surjective and hence of finite degree.
\end{thm}

\begin{proof}
Since the extension $\widetilde \pr_{2,7}$ is proper according to Theorem \ref{thm4.1}, it suffices 
to show that the map $\pr_{2,7}$ is generically finite.
Now $\pr_{2,7}$ is generically finite as soon as its differential at the
generic point $[f:\tC \ra C] \in \cR_{2,7}$ is injective. 
Let $f$ be given by the 7-division point $\eta$.
According to Lemma \ref{lem5.1}, the codifferential of $\Prym_{2,7}$
at $[f: \tC \ra C]$ is given by (the sum of) the multiplication of sections
$$
\mu_{C,\eta} :    \oplus_{j=1}^3 H^0(C, \omega_C \otimes \eta^{j} )
\otimes  H^0(C, \omega_C \otimes \eta^{7-j} )  \lra H^0(C, \omega^2_C ).
$$
Since  $\overline \cR_{2,7}$ is irreducible and $\widetilde \cR_{2,7}$ is open and dense in $\overline \cR_{2,7}$,
it suffices to show that the map $\mu_{X,\eta}$ is surjective at a point 
$(X,\eta)$ in the compactification $\overline \cR_{2,7}$,  even if $\pr_{2,7}$ is not defined at $(X,\eta)$. 
So if $\mu_{X,\eta}$ is surjective at this point, it will be surjective 
at a general point of $\cR_{2,7}$. Moreover, it suffices to show that $\mu_{X,\eta}$ is injective, since 
both sides of the map are of dimension 3.

Consider the curve 
$$
X = Y \cup Z,
$$ 
the union of two rational curves intersecting in 3 points ${q_1, q_2, q_3}$ which we 
can assume to be $[1,0], [0,1], [1,1]$ respectively. A line bundle 
$\eta_X= (\eta_Y, \eta_Z)$ on $X$ of degree 0 is uniquely determined by the gluing of the fiber over the nodes $\cO_{Y_{|q_i}}
\stackrel{\cdot c_i}{\ra }  \cO_{Z|_{q_i}}$, given by the multiplication  by a non-zero constant $c_i$. We may assume 
$c_3=1$ and since $\eta_X^7\simeq \cO_X$ we have $c_1^7=c_2^7=1$.   Notice that $\omega_{X_{|Y}} = \cO_Y(1)$ and $\omega_{X_{|Z}} = \cO_Z(1)$
and the restrictions  $\eta_Y, \eta_Z$ are trivial line bundles. 

From the exact sequence
$$
0 \ra  \cO_Z(-2) \ra \omega_X\otimes \eta_X^i \stackrel{\beta_i}{\ra} \cO_Y(1) \ra 0
$$
follows that $h^0(X, \omega_X \otimes \eta_X^i)=1 $ for $i=1, \dots, 6$. Moreover, since $H^0(Z, \cO_Z(-2))=0$, the map $\beta_i$
induces an inclusion  $H^0(X, \omega_X \otimes \eta_X^i) \hookrightarrow H^0(Y, \cO_Y(1))$ for $i=1, \dots, 6$. 

Therefore,  to 
study the injectivity of the map $\mu_{X,\eta}$,  it is enough to check whether the projection of $\oplus_{i=1}^{3}  H^0(X, \omega_X \otimes \eta_X^i) \otimes
H^0(X, \omega_X \otimes \eta_X^{7-i}) $ to 
$H^0(Y, \cO_Y(1)) \otimes H^0(Y, \cO_Y(1))$ is contained in the kernel of the multiplication map 
$$
H^0(Y, \cO_Y(1))  \otimes  H^0(Y, \cO_Y(1)) \lra  H^0(Y, \cO_Y(2)).
$$

We claim that the line bundle $\omega_X= (\omega_{|Y}, \omega_{|Z})$ is uniquely determined and 
one can choose the gluing $c_i$ at the nodes $q_i$ to be the multiplication by the same constant. To see this, first 
notice that, since $(X, \omega_X)$ is a limit linear series of canonical line bundles,  the nodes of $X$ are 
necessary Weierstrass points of $X$. Let $s_3 \in H^0(X, \omega_X(-2q_3)) $ be a section 
giving a trivialization of $\omega_Y$ and $\omega_Z$ away from $q_3$. For $i=1,2$, we have 
$\cO_Y (1)_{|q_i} \stackrel{s_3^{-1}}{\ra} {\cO_X}_{|q_i} \stackrel{s_3}{\ra} \cO_Z(1)_{|q_i}$, which implies that
$c_1=c_2$. Similarly, by using a section  in $ H^0(X, \omega_X(-2q_2)) $ one shows that $c_1=c_3$.
 
A section of  $\omega_{X_{|Y}} \otimes \eta_Y^i  \simeq \cO_Y (1)$ for $i=1,2,3$ is of the form $f_i(x,y)= a_ix + b_i y$, with $a_i, b_i $
constants. Suppose that the sections $f_i$ are in the image of the inclusion
$$
H^0(X, \omega_X \otimes \eta_X^i) \hookrightarrow  H^0(Y, \cO_Y(1)).
$$
 By evaluating the section at the points $q_i$ and using the gluing conditions one gets $a_i=c_1^i-1$
and $b_i = c_2^i-1$.  One obtains a similar condition for the image of the sections of
$H^0(X, \omega_{X} \otimes \eta^{7-i}_X )$ in $H^0(Y, \cO_Y(1))$.
Set $j=7-i$.  By multiplying the corresponding sections of $ \omega_X \otimes \eta_X^i$ and  $\omega_X \otimes \eta_X^j$ we have that an element
in the image of  $\mu_{X,\eta}$ is of the form  
$$
(2-c_1^i-c_1^j) x^2 +  (2-c_2^i-c_2^j) y^2 - (2- c_1^j- c_2^i + c_1^jc_2^i - c_1^i - c_2^j  + c_1^ic_2^j )xy.
$$
Hence, after taking the sum of such sections  for $i=1,2,3$ we conclude that there is a non-trivial element in the kernel of $\mu_{X,\eta}$ if and only if 
there is a non-trivial solution  for the linear system $A{\bf x} =0 $ with
$$
A= \left(
\begin{array}{ccc}
2-c_1-c_1^6 & 2- c_1c_2^6 -c_1^6c_2  & 2-c_2-c_2^6 \\
2-c_1^2-c_1^5 & 2- c_1^2c_2^5 -c_1^5c_2^2  & 2-c_2^2-c_2^5 \\
2-c_1^3-c_1^4 & 2- c_1^3c_2^4 -c_1^4c_2^3  & 2-c_2^3-c_2^4 
\end{array}
\right)
 $$ 
Clearly, if $c_i =1$ for some $i$ or  $c_1=c_2$ the determinant of $A$ vanishes. 
We compute
$$
\frac{1}{7}\det A= c_1^6(c_2^3 -c_2^5) +  c_1^5(c_2^6 -c_2^3) + c_1^4(c_2^2 -c_2) +  c_1^3(c_2^5 -c_2^6) + c_1^2(c_2 -c_2^4) +  
c_1(c_2^4 -c_2^2).
$$
Suppose that $c_i \neq 1$ and $c_2 = c_1^k$ for some $2\leq k\leq 6$.  Then a straightforward computation shows that $\det A \neq 0 $ 
if and only if $k=3$ or $k=5$. 

In conclusion, we can find a limit linear series $(X, \eta_X)$ with $\eta_X^7 \simeq \cO_X$, for suitable values of the $c_i$, 
such that the composition map in the commutative diagram
$$
\xymatrix{
\oplus_{i=1}^{3 } H^0(X, \omega_X \otimes \eta^{i} )\otimes  H^0(X, \omega_X \otimes \eta^{7-i} )\ar[r] \ar@{^(->}[d] &   
H^0(X, \omega^2_X ) \ar[d]^{\simeq} \\
H^0(Y, \cO_Y(1))  \otimes  H^0(Y, \cO_Y(1)) \ar[r]  & H^0(Y, \cO_Y(2))
}
$$ 
is an isomorphism.  
\end{proof}


\section{A complete fibre of $\widetilde \pr_{2,7}$}


For a special point of $\cB_D$ consider a smooth curve $E$ of genus 1. Then the kernel of the 
addition map
$$
X = X(E):= \ker (m: E^7 \ra E) \quad  \mbox{with} \quad m(x_1, \dots,x_7) = x_1 + \dots + x_7
$$ 
is an 
abelian variety of dimension 6, isomorphic to $E^6$. The kernel of the induced polarization of the
canonical principal polarization of $E^7$ is $\{ (x,\dots,x) : x \in E_7 \}$ which consists of $7^2$ 
elements. So the polarization on $X$ induced by the canonical polarization of $E^7$ is of type 
$D = (1,1,1,1,1,7)$.
Since the symmetric group $\cS_7$ acts on $E^7$ in the obvious way, 
$X$ admits an automorphism of 
order 7. Hence $X$ with the induced polarization is an element of $\cB_D$. 
To be more precise, the group $\cS_7$ admits exactly 120 subgroups of order 7. Hence to every elliptic curves there exist exactly 
120 abelian varieties $X$ as above with $G$-action. All of them are isomorphic to each other, since the corresponding subgroups 
are conjugate to each other according the Sylow theorems.
We want to determine the 
complete preimage $\widetilde \pr_{2,7}^{-1}(X)$ of $X$. We need some lemmas. For simplicity we denote $ \pr(f)$ the 
Prym variety of a covering $f$ in $\widetilde \cR_{2,7}$.

\begin{lem} \label{L6.1}
Let $f: \widetilde C \ra C$ be a covering satisfying $(**)$ with $g = 2$ such that $ \pr(f) \simeq X$. 
Then $C$ contains a node of index $1$.
\end{lem}

\begin{proof}
Suppose that either $C$ is smooth or all the nodes of $C$ are of index 7. 
Then the  exact sequence \eqref{eq2.2} gives an isogeny $j: \pr(f) \ra \pr(\widetilde f)$ onto the Prym 
variety  of the normalization $\widetilde f$ of $f$.  Actually, in the smooth case $j$ is an isomorphism and, 
if there is a node of index $1$, then the kernel $\widetilde{T}_7$ would be positive dimensional.   
The isomorphism $\pr(f) = X$ implies that the kernel of $j$ is of the form $\{(x, \dots,x)\}$ with $x \in X_7$. 
Hence the action of the symmetric  group $S_7$
on $X$ descends to a non-trivial action on $ \pr(\widetilde f)$. 

We can extend this action to $J \widetilde N$ by combining it with the identity on $JN$. Namely, 
$J \widetilde N \simeq (JN \times  \pr(\widetilde f)/H$ where $H$ is constructed as follows. Let 
$\langle \eta \rangle \subset JN_7$ be the subgroup defining the covering $\widetilde f$ and
$H_1 \subset JN_7$ be its orthogonal complement with respect to the Weil pairing. Then
$H = \{(\alpha,-f^*\alpha): \alpha \in H_1\}$. Since $f^*H_1 = \{(x,\dots,x): x \in E_7 \} \subset 
 \pr(f)$, we get an $S_7$-action on $J \widetilde N$ which is clearly  non-trivial.

If $C$ is smooth $\widetilde{N} \simeq \widetilde{C}$. On the other hand, if all the nodes of $C$ are of 
index 7, $\widetilde{N}$ consists of at least two components.  In any case, for each component  $\widetilde{N}_i$
of $\widetilde{N}$ we have 
$$
\# \Aut(J\widetilde N_i) \geq \frac{1}{2} \# S_7 = 2520.
$$

Moreover,  according to a classical theorem of Weil, $\Aut \widetilde N_i$ embeds into $\Aut 
J\widetilde N_i$ with quotient of order $\leq 2$. So
$$
\# \Aut \widetilde N_i \geq \frac{1}{2} \# \Aut (J \widetilde N_i).
$$
On the other hand, $\widetilde N_i$ is a smooth curve of genus $\leq 8$. So Hurwitz' theorem implies 
that 
$$
\# \Aut (\widetilde N_i) \leq 84 \cdot (8-1) = 588.
$$
Together, this gives a contradiction.
\end{proof}

\begin{lem}  \label{L6.2}
Let $f:\widetilde C \ra C$ be a covering satisfying $(**)$ such that $C$ has a component 
containing nodes of index $1$ and $7$. Then any node of index $1$ is the intersection with
another component of $C$.
\end{lem}

\begin{proof}
Suppose $x$ and $y$ are nodes of $C$ in a component $C_i$
of index 1  and 7 respectively. Then the preimage $f^{-1}(C_i)$ is a component, since over $y$
the map $f$ is totally ramified. Since $f^{-1}(x)$ consists of $7$ nodes, the equality $n_1 = c_1$
implies that $x$ is the intersection of 2 components. 
\end{proof}

\begin{thm} \label{thm6.3}
Let $X = \ker(m: E^7 \ra E)$ be a polarized abelian variety as above. The fibre $\widetilde \pr_{2,7}^{-1}(X)$ consists of the following $4$ types 
of elements of $\widetilde \cR_{2,7}$.

{\em (i)} $C = E/{p \sim q}$ and $\widetilde C = \sqcup _{i=1}^7 E_i /{p_i \sim q_{i+1}}$
with $E_i \simeq E$ for all $i$ and $q_8 = q_1$
and we can enumerate in such a way that the preimages 
of $p$ and $q$  are $p_i, q_i \in E_i$. 

{\em (ii)} $C = E_1 \cup_p E_2$ consists of $2$ elliptic curves intersecting in one point $p$.
Then up to exchanging $E_1$ and $E_2$ we have: $\widetilde C$ consists of an elliptic curve $F_1$, 
which is a 7-fold cover of $E_1$ and 7 copies of $E_2 \simeq E$ not intersecting each other and 
intersecting $F_1$ each in one point. 

{\em (iii)} $C = E_1 \cup_p E_2$ with $E_2$ elliptic and $E_1$ rational with a node at $q$.
Then $E_2 \simeq E$ and $f^{-1}(E_2)$ consists of $7$ disjoint curves all isomorphic to $E$ and 
$f^{-1}(E_1)$ is a rational curve with one node lying $7:1$ over $E_1$ and intersecting each 
component of $f^{-1}(E_2)$ in a point over $p$.  

{\em (iv)} $C = E_1 \cup_p E_2$ as in {\em (iii)} and $\widetilde C$ and \'etale $G$-cover over $C$.
\end{thm}

We call the coverings of the theorem {\it of type} (i) , (ii), (iii) and (iv) respectively.

\begin{proof}
There are 7 types of stable curves of genus 2. We determine the coverings $f:\widetilde C \ra C$
in $ \widetilde \pr_{2,7}^{-1}(X)$ in each case separately. 

1) There is no \'etale $G$-cover $f:\widetilde C \ra C$ of a smooth curve $C$ of genus $2$ such that
$P = \Prym(f)  \simeq X.$ This is a direct consequence of Lemma \ref{L6.1}.

2) If $C = E/{p\sim q}$ 
then the singular point of $C$ is of index 1 and $f: \widetilde C \ra C$ a $G$-covering satisfying 
$(**)$ such that $P = \Prym(f) \simeq X$. Then
$$
\widetilde C = \sqcup _{i=1}^7 E_i /({p_i \sim q_{i+1}})
$$
with $E_i \simeq E$ and we enumerate in such a way that the preimages 
of $p$ and $q$  are $p_i$ and $q_i$ with
 $q_8 = q_1$. In this case $\Prym(f) \simeq X$.

{\it Proof:}
According to Lemma \ref{L6.1} the node is necessarily of index 1 and thus the map 
$\widetilde f: \widetilde C \ra C$ is \'etale.
The exact sequence \eqref{eq2.2} gives an isomorphism $P \simeq R$ with $R$ the Prym variety
of the map  $\widetilde f$. Clearly we can enumerate the components of $\widetilde N$ in such a way 
$\widetilde C$ is as above and $R$ is the kernel of 
the map $m: \times_{i=1}^7 E \ra E$, i.e. $R \simeq X$. We are in case (i) of the theorem.

3) There is no rational curve $C$ with $2$ nodes admitting a $G$-cover $f: \widetilde C \ra C$ 
satisfying $(**)$ such that $P = \Prym(f) \simeq X$.

{\it Proof:}
Suppose there is such a covering. By Lemmas \ref{L6.1} and \ref{L6.2} both nodes are of index 1 
and hence the map $\widetilde C \ra C$ is \'etale. Then all components 
of $\widetilde C$ are rational. This implies that $P \simeq {\CC^*}^6$ is not an abelian variety,
a contradiction.

4) Let $C = E_1 \cup_p E_2$ consist of $2$ elliptic curves intersecting in one point and 
$f: \widetilde C \ra C$ be covering satisfying $(**)$ such that $P = \Prym(f) \simeq X$.
Then, up to exchanging $E_1$ and $E_2$, we have that $\widetilde C$ consists of an elliptic curve $F_1$
which is a 7-fold cover of $E_1$, and 7 copies of $E_2 \simeq E$ not intersecting each other and intersecting $F_1$, 
each in one point. So $X = \ker(m: E_2^7 \ra E_2)$.

{\it Proof:}
By Lemma \ref{L6.1} the node is of index 1 and the map $f:\widetilde C \ra C$ is \'etale. 
since there is no
connected graph with 14 vertices and 7 edges,  we are necessarily in case (ii) of the theorem.

5) Suppose $C = E_1 \cup_p E_2$ with components $E_2$ elliptic and $E_1$ rational
with a node $q$ and $f: \widetilde C \ra C$ a $G$-covering satisfying 
$(**)$ such that $P = \Prym(f) \simeq X$.  Then $E_2 \simeq E$ and either $f: \widetilde C \ra C$ is \'etale and connected or 
$\widetilde C$ consists of 7 components all isomorphic to $E$
and a rational component $F_2$ over $E_2$ totally ramified exactly over $q$ and intersecting each
$E_i$ exactly in one point lying over $p$. So $\pr(f) \simeq X$.

{\it Proof:}
According to Lemma \ref{L6.1} at least one node of $C$ is of index 1.  
Suppose first that both nodes are of index 1. Then clearly $f$ is \'etale and we are in case (iv) 
of the theorem.
If only one node is of index 1, then according to Lemma \ref{L6.2}, $q$ is of index 7 and $p$ of index 1. 
This gives case (iii) of the theorem.

6) There is no curve $C$ consisting of 2 rational components intersecting in one point $p$ admitting a $G$-cover
$f: \widetilde C \ra C$ satisfying $(**)$ such that $P = \Prym(f) \simeq X$.

{\it Proof:}
According to Lemmas \ref{L6.1} and \ref{L6.2} the node $p$ is of index 1 and the nodes $q_1$ and $q_2$ of the rational 
components of $C$ are of index 1 or 7. By the Hurwitz formula 
all components of $\tC$ are rational. This implies $\pr(f) \simeq {\CC^*}^6$ contradicting $(**)$.

7) If $C$ is the union of 2 rational curves intersecting in 3 points, there is no cover $f: \widetilde C \ra C$ satisfying $(**)$
such that $P = \Prym(f) \simeq X$.

{\it Proof:}
By Lemma \ref{L6.1} at least one of the 3 nodes of $C$ is of index 1. So $\widetilde C$ consists of 
at least 8 components.  But then the other nodes also 
are of index 1, because if one node is of index 7, the curve $\widetilde C$ consists of 2 components
only. 
Hence all 3 nodes  nodes are of index 1. But then all components of $\widetilde C$ are rational. So 
$P=\pr(f)$ cannot be an abelian variety contradicting $(**)$.

Together, steps 1) to 7) prove the theorem.
\end{proof}

\begin{figure}
\begin{center}
\includegraphics[width=12.7cm, height=10cm]{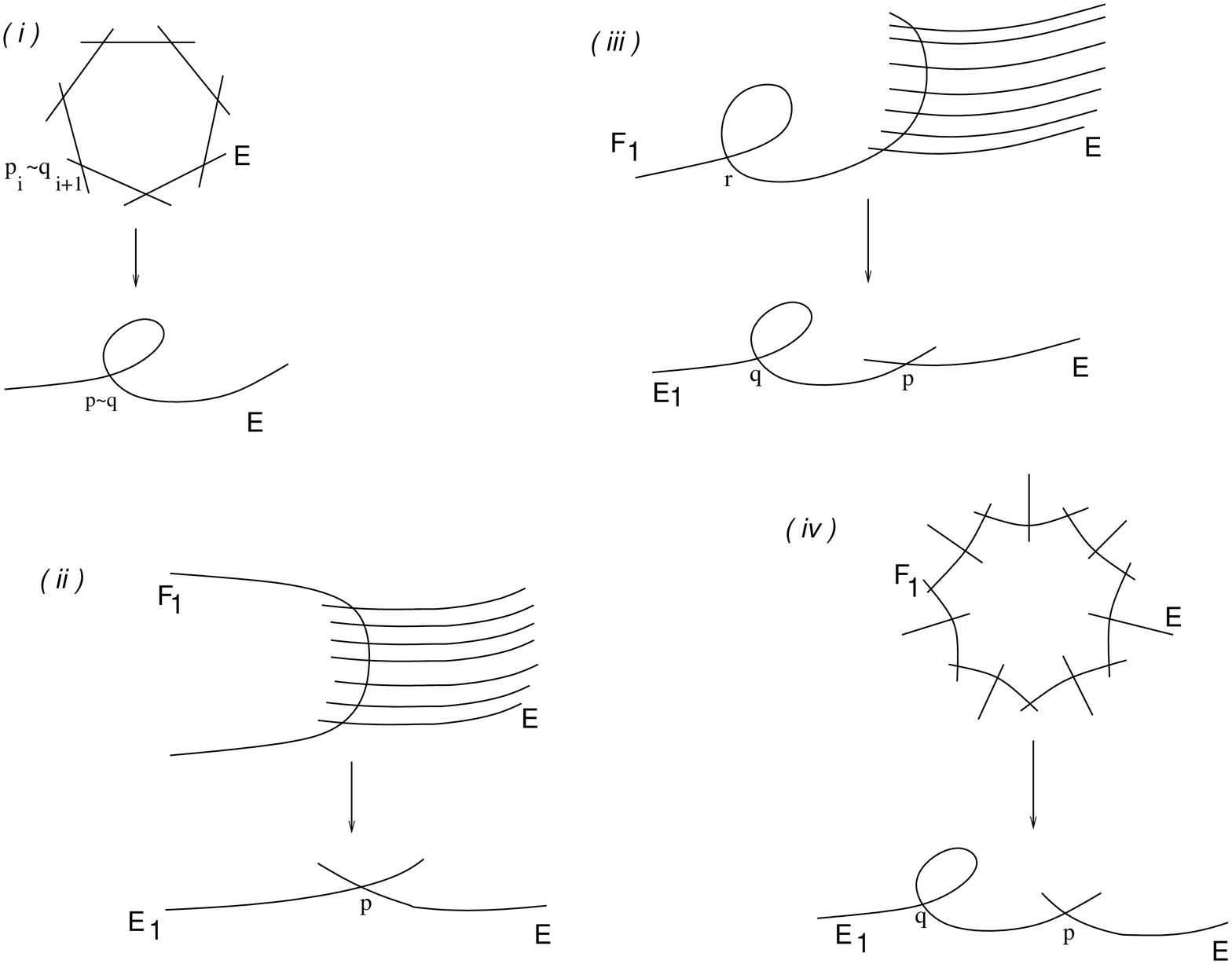}
\caption{Admissible coverings on the fiber of $X(E)$}
\end{center}

\end{figure}

Varying the elliptic curve $E$, we obtain a one dimensional locus $\cE \subset \cB_D$ consisting of 
the polarized abelian varieties $X(E)$ with  $G$-action as above. Let $\cS$ denote the preimage of $\cE$ under the extended 
Prym map $\widetilde \pr_{2,7}: \widetilde \cR_{2,7} \ra \cB_D$.

\begin{prop} \label{prop6.4}
The scheme $\cS$ is the union of $2$ closed subschemes
$$
\cS = \cS_1 \cup \cS_2,
$$
where $\cS_1$ (respectively $\cS_2)$ parametrizes coverings of type {\em (i)} and {\em (iv)} 
(respectively of types {\em (ii), (iii)} and {\em (iv)}). In particular they intersect exactly in the points 
parametrizing coverings of type {\em (iv)}.
For a general elliptic curve $E$ and $f$  a covering in $\cS$ with 
$X(E) = \widetilde \pr_{2,7}(f)$ we have 

If $f$ is of type {\em (i)}, then $\widetilde \pr_{2,7}^{-1}(X(E))$ is isomorphic to a finite covering of $E \setminus p$ where 
$p$ maps to the singular point of $C$. 

If $f$ is of type {\em (ii)}, then $\widetilde \pr_{2,7}^{-1}(X(E))$ is isomorphic to a finite covering of the moduli space of elliptic curves. 

There is only one covering $f$  of type {\em (iii)} and {\em (iv)} in the fiber $\widetilde \pr_{2,7}^{-1}(X(E))$.

\end{prop}

\begin{proof}
It is known that for 2 elliptic curves $E_1 \neq E_2$ we can have $X(E_1) \simeq X(E_2)$ as abelian 
varieties, but not necessarily as polarized abelian varieties. Hence $X(E)$ determines $E$ (which can be seen also from Theorem \ref{thm6.3}). 

We claim that the coverings of type (iv) are contained in $\cS_1$ and $\cS_2$ whereas the coverings 
of type (iii) are contained in $\cS_2$ only: it is known that a curve $\widetilde C$ degenerate to a 
curve $\widetilde C'$ of some other type if and only if the the dual graph of $\tC'$ can be contracted
to the 
dual graph of $\tC$. On the other hand, the locus of curves covering some curve of genus $\geq 2$
of some fixed degree is closed in the moduli space of curves. Now considering the dual graphs of the 
curves $\tC$ of the coverings of the different types gives the assertion.

Hence it suffices to show 
the assertions about the fiber of $X(E)$ under the map $\widetilde \pr_{2,7}$.
In case (i) we have $C = E/p_1 \sim p_2$. We can use the translations of $E$ to fix $p_1$ and 
then $p_2$ is free, which gives the assertion, since there are only finitely many \'etale coverings of 
$C$.

In case (ii) we have $C = E_1 \cup_p E_2$ where $E_1$ is an arbitrary elliptic curves and 
$E_2 \simeq E$. Since $p$ 
may be 
fixed with an isomorphism of $E$ and $E_2$, this give the isomorphism of 
$\widetilde \pr_{2,7}^{-1}(E)$
with a finite covering of the moduli space of elliptic curves, again since there are only finitely many 
coverings $\tC$ of type (ii) of $C$.

Finally in cases (iii) and (iv),  the 3 points of the normalization of $E_1$ given by $p$ and the 2 preimages of the node,
that we can assume to be $1,0,$ and $\infty$ respectively,  determine the curve $C$ uniquely. For the type (iii) the induced map on the 
normalization of $F_1$ is a 7:1 map $h: \PP^1 \ra \PP^1$ totally ramified at 2 points, that we assume to be $\infty$ and $0$. So $h$ it can be 
expressed as a polynomial in one variable of degree 7, with  vanishing order 7 at 0 and such that $h(1)=1$, that is $h(x)=x^7$. Then 
the map $h$, and hence the covering is uniquely determined.  For a covering of type (iv) over $C$ we consider 7 copies of $\PP^1$ and
where the point $1$ on every rational component is identified to the point $\infty$ of other rational component and we attach elliptic  
curves isomorphic to $E_2$ at each point $0$.  The number of \'etale coverings is number of subgroups of order 7 in $JE_1[7] \simeq \ZZ/7\ZZ $,
(the 7-torsion points in the nodal curve $E_1$ are determined by a 7-rooth of unity). So there is only one of such covering up to isomorphism. 
\end{proof}


\section{The codifferential on the boundary divisors}


In this section we will give bases of the Prym differentials and an explicit description of 
the codifferential of the Prym map. 

Let $f:\widetilde C \ra C$ be a covering corresponding to a point of $\cS$. We want to compute the 
rank of the codifferential of the Prym map $\pr_{2,7}: \widetilde \cR_{2,7} \ra \cB_D$ 
at the point $[f:\widetilde C \ra C] \in \widetilde \cR_{2,7}$. According to \cite{ds} this codifferential is the map
$$
\cP^*: S^2(H^0(\widetilde C, \omega_{\widetilde C})^-)^G \ra H^0(C,\Omega_C \otimes \omega_C).
$$ 
where $\Omega_C$ is the sheaf of K\"ahler differentials on $C$ and $S^2(H^0(\widetilde C, \omega_{\widetilde C})^-)^G$ is the 
cotangent space to $\cB_D$ at the Prym variety of the covering $f: \widetilde C \ra C$. 
If $j: \Omega_C \ra \omega_C$ denotes the canonical map we first compute the rank of the 
composed map
$$
S^2(H^0(\widetilde C, \omega_{\widetilde C})^-)^G \stackrel{\cP^*}{\lra} H^0(C,\Omega_C \otimes \omega_C) \stackrel{j}{\lra}
H^0(C, \omega^2).
$$
Suppose first that $f$ is of type (i), (ii) or (iv). In these cases the covering $f$ is \'etale and hence given by a $7$-division point
$\eta$ of $JC$. According to the analogue of \cite[equation (3.4)]{lo} we have
\begin{equation} \label{eq1}
H^0(\widetilde C, \omega_C)^- = \oplus_{i=1}^6 H^0(C,\omega_C \otimes \eta^i)
\end{equation}
and hence
\begin{equation} \label{eq2}
S^2(H^0(\widetilde C, \omega_C)^-)^G = \oplus_{i=1}^3 ( H^0(C, \omega_C \otimes \eta^i) \otimes
H^0(C, \omega_C \otimes \eta^{7-i})).
\end{equation}
Using this, the above composed map is just the sum of the cup product map
\begin{equation}  \label{eq3}
\phi:  \oplus_{i=1}^3 (H^0(C, \omega_C \otimes \eta^i) \otimes
H^0(C, \omega_C \otimes \eta^{7-i})) \lra H^0(C, \omega_C^2)
\end{equation}
whose rank we want to compute first.

We shall give a suitable basis for the space of Prym differentials.
First we consider a covering of type (i) constructed as follows.
Let $E$ be a smooth curve of genus 1 and $q \neq q'$ be two fixed points of $E$. Then
$$
C := E/ q \sim q'
$$
is a stable curve of genus 2 with normalization $n:E \ra C$ and node $p:= n(q) = n(q')$.
Let $f:\widetilde C \ra C$ be a cyclic \'etale covering with Galois group $G = \langle \sigma \rangle 
\simeq \ZZ/7\ZZ$. The normalization $ \widetilde n: \widetilde N \ra \widetilde C$ consists of 7
components $N_i \simeq E$ with $\sigma(N_i) = N_{i+1}$ for $i = 1, \dots 7$ with $N_8 = N_1$. Let $q_i$ and $q'_i$ the elements of 
$N_i$ corresponding to $q$ and $q'$.
Then $\widetilde n(q_i) = \widetilde n(q'_{i+1}) =:p_i$ for $i= 1,\dots 7$ with $q'_8 = q'_1$. Clearly 
$\sigma(p_i) = p_{i+1}$ for all $i$.

Recall that $\omega_{\widetilde C}$ is the subsheaf of 
$\widetilde n_*\left( \cO_{\widetilde N}\sum(q_i + q'_i) \right)$ consisting of (local) sections 
$\varphi$ which considered as sections of $\cO_{\widetilde N}\sum(q_i + q'_i) $ satisfy the condition
$$
Res_{q_i}(\varphi) + Res_{q'_{i+1}}(\varphi) = 0
$$
for $i = 1, \dots 7$. 
Here we use the fact that $\omega_{\widetilde N} = \cO_{\widetilde N}$.
Consider the following elements of 
$H^0(\widetilde C,\widetilde n_*\left( \cO_{\widetilde N}\sum(q_i + q'_i) \right)$, regarded as 
sections on $\widetilde N$: 
$$
\omega_1:= \left\{ \begin{array}{l}
                             \mbox{nonzero section of} \; \cO_{N_1}(q_1+q'_1) \; \mbox{vanishing at} 
                             \; q_1 \; \mbox{and} \; q'_1\\
                              0 \; \mbox{elsewhere}
                               \end{array}  \right.
$$
and for $2, \dots 7$,
$$
\omega_i := (\sigma^{-i})^* (\omega_1).
$$
Note that $\omega_i$ is nonzero on $N_i$ vanishing in $q_i$ and $q'_i$ and zero elsewhere. 

Now we construct similar differentials for coverings of type (ii).
Let
$$ 
C = E_1 \cup_p E_2
$$
consist of 2 elliptic curves $E_1$ and $E_2$ intersecting transversally in one point $p$ and let $f: \tC \ra C$ be a covering of type (ii).
So $\tC$ consists of an elliptic curve $F_1$, which is an \'etale cyclic cover of $E_1$ of degree 7 and 7 disjoint curves $E_2^1, \dots,  E_2^7$ 
all isomorphic to $E_2$. The curve $E_i$ intersects $F_1$ transversally in a point $p_i$, such that the group $G$ permutes the curves $E_i$ and the points
$p_i$ cyclicly i.e. $\sigma(E_2^i) = E_2^{i+1}$ and $\sigma(p_i) = p_{i+1}$ with $E_2^8 = E_2^1$ and $p_8 = p_1$.

Let $\widetilde n:\widetilde N \ra\tC$ denote the normalization map. Then $\widetilde N$ is the disjoint union of the 8 elliptic curves $F_1, E_2^1, \dots, E_2^7$.
We denote the point $p_i$ be the same letter when considered as a point of $F_1$ and $E_2^i$. Consider the line bundle
$$
L = \cO_{F_1}(p_1 + \cdots + p_7) \sqcup \cO_{E_2^1}(p_1) \sqcup \cdots \sqcup \cO_{E_1^7}(p_7).
$$
Then $\omega_{\tC}$ is the subsheaf of $\widetilde n_*(L)$ consisting of (local) sections $\varphi$, 
which considered as sections of $L$ satisfy the condition
\begin{equation} \label{eq9.4}
Res_{p_i}|_{F_1}(\varphi) + Res_{p_i}|_{E_2^i}(\varphi) = 0
\end{equation} 
for $i = 1, \dots, 7$.  Consider the following section of $\widetilde n_*(L)$ regarded as section on 
$\widetilde N$: 
$$
\omega_1:= \left\{ \begin{array}{l}
                             \mbox{nonzero sections of} \; \cO_{F_1}(p_1) \; \mbox{and } \; \cO_{E_2^1}(p_1)
                             \;   \mbox{satisfying \eqref{eq9.4}  at} \; p_1,\\
                              0 \; \mbox{elsewhere}
                               \end{array}  \right.
$$
and for $2, \dots 7$, the sections 
$$
\omega_i := (\sigma^{-i})^* (\omega_1).
$$
Thus $\omega_i$ is nonzero on $F_1(p_i) \sqcup E_2^i(p_i)$,  vanishing in $p_i$ and zero 
elsewhere. 




We construct the analogous differentials for the covering $f$ of type (iv),  which is uniquely determined according to 
Proposition \ref{prop6.4}. So let
$$
C = E_1 \cup_p E_2
$$
with $E_2$ elliptic and $E_1$ a rational curve with one node $q$ and let $f: \tC \ra C$ be the covering of type (iv). So $\tC$ 
consists of 14 components
$F_1, \dots, F_7$ isomorphic to $\PP^1$ with $f|_{F_i}: F_i \ra E_1$ the normalization and $E_2^1, \dots, E_2^7$ all isomorphic to $E_2$ with 
$f|_{E_2^i}: E_2^i \ra E_2$ the isomorphism. Then $E_2^i$ intersects $F_i$ in the point $p_i$ lying over $p$ and no other component of $\tC$. 
If $q_i$ and $q_i'$ are the points 
of $F_i$ lying over $q$, the $F_i$ and $F_{i+1}$ intersect transversally in the points $q_i$ and $q'_{i+1}$ for $i= 1, \dots, 7$ where $q'_8 = q_1$. 
The group $G$ permutes the components and points cyclicly, i.e. $\sigma(F_1) = F_{i+1}$ and similarly for $E_2^i, p_i, q_i$ and $q_i'$.

The normalization $\widetilde n: \widetilde N \ra \tC$  of the curve $\tC$ is the disjoint union of the
 components $F_i$ and $E_2^i$. We denote also the point $p_i$ 
by the same letter when considered as a point of $F_i$ and $E_2^i$. Consider the following line
 bundle on $\widetilde N$
$$
L = \bigsqcup_{i=1}^7 \cO_{F_i}(q_i + q_i' + p_i) \sqcup \bigsqcup_{i=1}^7 \cO_{E_2^i}(p_i).
$$
Then $\omega_{\tC}$ is the subsheaf of $\widetilde n_*(L)$ consisting of (local) sections $\varphi$,
 which viewed as sections of $L$ satisfy the conditions
\begin{equation} \label{eq9.5}
Res_{p_i}|_{F_i}(\varphi) + Res_{p_i}|_{E_2^i}(\varphi) = 0 \quad 
\mbox{and} \quad  Res_{q_i}|_{F_i}(\varphi) + Res_{q_{i+1}'}|_{F_{i+1}}(\varphi) = 0 
\end{equation} 
for $i = 1, \dots, 7$.  Let $\omega_1$ be the section of $\widetilde n_*( L)$ considered as section on  
$\widetilde N$ defined as follows: 
$$
\omega_1:= \left\{ \begin{array}{l}
                             \mbox{nonzero sections of} \; \omega_{F_1}(q_1 + q_1' +p_1) \; \mbox{and} \;
                              \cO_{E_2^1}(p_1) 
                             \; \mbox{ satisfying \eqref{eq9.5} at} \; p_1 \;  \mbox{and} \\
                             \mbox{vanishing at} 
                             \;q_1  \;\mbox{and} \; q_1';\\
                              0 \; \mbox{elsewhere}
                               \end{array}  \right.
$$
and for $2, \dots, 7$ define
$$
\omega_i := (\sigma^{-i})^* (\omega_1).
$$
Note that up to a multiplicative constant there is exactly one such section $\omega_1$, since 
$h^0(\omega_{F_1}(q_1 + q_1' +p_1)) = 2$ and $h^0( \cO_{E_2^1}(p_1)) = 1$.

Finally, we consider coverings of type (iii). Let $C = E_1 \cup_p E_2$ 
as for the covering of type (iv) above
and let $f : \tC \ra C$ be a covering of type (iii). So $\tC$ 
consists of a rational curve $F_1$ with a node $r$
lying over the node $q$ of $C$ and 7 components $E_2^1, \dots, E_2^7$ all isomorphic to $E_2$.
Then $E_2^i$ intersects $F_1$ in the point $p_i$ lying over $p$ and intersects no other component of 
$\tC$. The group $G$ acts on $F_1$ with only fixed point $r$ and permutes the $E_i^2$ and $p_i$ 
cyclically as above. We use the following partial normalization $\widetilde n: \widetilde N \ra \tC$ of 
$\tC$:
$$
\widetilde N : = F_1 \sqcup \bigsqcup_{i=1}^7  E_2^i.
$$
Consider the following line bundle on $\widetilde N$
$$
L = \cO_{F_1}(p_1 + \cdots + p_7) \sqcup \bigsqcup_{i=1}^7 \cO_{E_2^i}(p_i) .
$$
Since the canonical bundles of $F_1$ and $E_2^i$ are trivial, it is clear that $\omega_{\tC}$ is the 
subsheaf of $\widetilde n_*(L)$  consisting of (local) sections $\varphi$, which regarded as sections of $L$ satisfy the relations
\begin{equation} \label{eq8.6}
Res_{p_i}|_{F_1}(\varphi) + Res_{p_i}|_{E_2^i}(\varphi) = 0
\end{equation}
for $i = 1, \dots ,7$. As before, define a section $\omega_1$ of $\widetilde n_*(L)$ considered as a section of 
$\widetilde N$:
$$
\omega_1:= \left\{ \begin{array}{l} 
\; \mbox{nonzero sections of} \;
\cO_{F_1}(p_1 + \cdots + p_7)  \; \mbox{and}  \; \cO_{E_2^i}(p_1) \;
\mbox{vanishing at} \; p_1, \dots, p_7;\\
0 \; \mbox{ elsewhere.}
\end{array}
\right.
$$
and define the sections $\omega_i$, for  $i=2, \ldots , 7$, as in the previous cases.
Note that up to a multiplicative constant there is exactly one such section $\omega_1$.

From now $f:\tC \ra C$ will be a covering of type (i)-(iv) as above. We fix a primitive 7-th root of unity, for example $\rho := e^{\frac{2 \pi i}{7}}$
and define for $i = 0, \dots ,6$ the section
$$
\Omega_i := \sum_{j=1}^7 \rho^{ij} \omega_j.
$$

Clearly $\Omega_i$ is a global section of $L$ which defines a section of $\omega_{\tC}$ which we 
denote with the same symbol.

\begin{lem} \label{basis}
$\sigma^* (\Omega_i) = \rho^i \Omega_i$  for $i = 0, \dots,6$. In particular, $\Omega_0 \in H^0(\widetilde C, \omega_{\widetilde C})^+$ 
and $\{ \Omega_1, \dots, \Omega_6 \}$ is a basis of $H^0(\widetilde C, \omega_{\widetilde C})^-$.
\end{lem}

\begin{proof}
The first assertion follows from a simple calculation using the definition of $\omega_i$. So clearly   $\Omega_0 \in 
H^0(\widetilde C, \omega_{\widetilde C})^+$ and $ \Omega_i \in H^0(\widetilde C, \omega_{\widetilde C})^-$ for $i=1, \ldots,6 $.
Since  $\Omega_1, \dots, \Omega_6 $ are in different eigenspaces of $\sigma$, they are linearly independent 
and since $H^0(\widetilde C, \omega_{\widetilde C})^-$ is of dimension 6, they form a basis.
\end{proof}

\begin{rem}
 In cases (i), (ii) and (iv) $H^0(C, \omega_C \otimes \eta^{7-i}) $ is the eigenspace of $\sigma^i$ and $\Omega_i$ is a generator 
for $i=1, \ldots, 6$. 
\end{rem}


\begin{prop}\label{codifDiv}
The map 
$$
\phi: S^2 (H^0(\tC, \omega_{\tC})^-)^G  \lra H^0(C, \omega_C^2)
$$
is of rank $1$.
\end{prop}


\begin{proof}
We have to show that the kernel of $\phi$ is 2-dimensional. 
A basis of $S^2 (H^0(\tC, \omega_{\tC})^-)^G $
 is given by $\{ \Omega_1 \otimes \Omega_6, \Omega_2 \otimes \Omega_5, \Omega_3 \otimes \Omega_4 \}$. So let $a,b,c$ be complex numbers with
$$
\phi(a \Omega_1 \otimes \Omega_6 + b \Omega_2 \otimes \Omega_5 + c \Omega_3 \otimes \Omega_4) = 0.
$$
Define for $i = 1, \dots,7$,
$$
\psi_i := \sum_{i=1}^7 \omega_j \otimes \omega_{j+i-1}.
$$
An easy but tedious computation gives
$$
\Omega_1 \otimes \Omega_6 = \psi_1 + \rho \psi_7 + \rho^2 \psi_6 + \rho^3 \psi_5+ \rho^4 
\psi_4 + \rho^5 \psi_3 + \rho^6 \psi_2,
$$
$$
\Omega_2 \otimes \Omega_5 = \psi_1 + \rho \psi_4 + \rho^2 \psi_7 + \rho^3 \psi_3 + \rho^4 
\psi_6 + \rho^5 \psi_2 + \rho^6 \psi_5,
$$
$$
\Omega_3 \otimes \Omega_4 = \psi_1 + \rho \psi_3 + \rho^2 \psi_5 + \rho^3 \psi_7 + \rho^4 
\psi_2 + \rho^5 \psi_4 + \rho^6 \psi_6.
$$
So we get
\begin{eqnarray*}
0 &=& \phi( (a+b+c)\psi_1 +(a\rho^6 + b \rho^5 + c \rho^4)\psi_2 + (a \rho^5 + b \rho^3 + 
c \rho) \psi_3 + (a \rho^4 + b \rho + c \rho^5) \psi_4\\
&&  \hspace{1.5cm}+(a \rho^3 + b \rho^6 +c \rho^2) \psi_5 + (a \rho^2 + b \rho^4 + c \rho^6)\psi_6 + (a \rho + b \rho^2 + c \rho^3) \psi_7)\\
&=& (a+b+c)(\omega_1^2 + \cdots + \omega_7^2) \\
&& \hspace{1.5cm}+[a(\rho+\rho^6) + b(\rho^2 + \rho^5) + c (\rho^3 + \rho^4)]\sum_{j=1}^7 \omega_j \omega_{j+1}\\
&& \hspace{1.5cm}+[a(\rho^2+\rho^5) + b(\rho^3 + \rho^4) + c (\rho + \rho^6)]\sum_{j=1}^7 \omega_j \omega_{j+2}\\
&& \hspace{1.5cm}+[a(\rho^3+\rho^4) + b(\rho + \rho^6) + c (\rho^2 + \rho^5)]\sum_{j=1}^7 \omega_j \omega_{j+3}.
\end{eqnarray*}
This section is zero if and only if its restriction to any component is zero. Now the restriction to $N_i$ for all $i$ gives
\begin{eqnarray*}
0 &=& a+b+c +(6a +6b +6c)\sum_{j=1}^6 \rho^j = -5(a+b+c).
\end{eqnarray*}
So $\phi(a \Omega_1 \otimes \Omega_6 + b \Omega_2 \otimes \Omega_5 + c \Omega_3 \otimes \Omega_4) = 0$ if and only if $a+b+c = 0$.
Hence the kernel of $\phi$ is of dimension 2, which proves the proposition.
\end{proof}

The proposition \ref{codifDiv} shows that the codifferential map along the divisors $\cS_1$ and $ \cS_2$ 
in Proposition \ref{prop6.4} is not  surjective. In fact, as we will see later, the 
kernel of $\phi$ coincides with the conormal bundle of the image of these divisors  in $\cB_D$.
In order to compute the degree we will perform a blow up along these divisors.  

Let $\cE \subset \cB_D$ denoted the one dimensional locus consisting of the abelian varieties 
which are of the form  $X = \ker  (m: E^7 \ra E) \quad  \mbox{with} \quad m(x_1, \dots,x_7) = x_1 + \dots + x_7$
for a given elliptic curve $E$. As we saw in Section 7, the induced polarization is of type $D$.   
Note that $\cE$ is a closed subset of $\cB_D$. The aim is to compute the degree of $\Pr_{7,2}$ above a point $X \in \cE$.
We denote by $\cS \subset \widetilde {\cR}_{2,7}$ the inverse image of $\cE$ under $\widetilde \Pr_{2,7}$. 
According to Proposition \ref{prop6.4}, $\cS$ is a divisor consisting of 2 irreducible components in the boundary 
$\widetilde{\cR}_{2,7} \setminus \cR_{2,7}$. We have
$\cS  = \cS_1 \cup  \cS_2$
where a general point of $ \cS_1$, respectively of $ \cS_2$, corresponds to the $G$-covers with base an irreducible nodal curve of genus 1,
respectively a product of elliptic curves intersecting in a point. Moreover, for any fixed elliptic curve $E$,
$\cS_1$ and $\cS_2$ intersect in the unique point given by the covering of type (iv).\\ 

As in \cite{ds}, we blow up $\cB_D$ along $\cE$ and obtain the following commutative diagrams:
$$
\xymatrix{
\widetilde{\cS} \ar[d] \ar[r]^{\widetilde{\cP}} &  \widetilde{\cB}_D  \ar[d]  \\
\widetilde {\cR}_{2,7} \ar[r]^{\widetilde \pr_{2,7}} &  {\cB}_D
}
\hspace{2cm}
\xymatrix{
\widetilde{\cS} \ar[d]_{\simeq} \ar[r]^{\widetilde{\cP}} &  \widetilde{\cE}  \ar[d]^{\PP^1}  \\
\cS \ar[r]^{\widetilde \pr_{2,7}} &  \cE
}
$$
where $\tilde{\cS}$ and $\tilde{\cE}$ are the exceptional loci. 

Lemma I.3.2 of \cite{ds} guarantees that the local degree of $\widetilde{\pr}_{2,7}$ along a component of $\cS $ equals the degree of the 
induced map on the exceptional divisors $\widetilde{\cP}_{|\cS_i}: \cS_i \ra \widetilde{\cE} $ if the codifferential map  $\cP^*$ is 
surjective on the respective  conormal bundles. Recall that the fibers of the conormal bundles at the point $X$ are given by
$$
\cN^*_{X, \cE / \cB_D} = \ker ( T_X^*\cB_D  \ra T_X^* \cE ) 
$$
$$
\cN^*_{(C, \eta), \cS_i /  \widetilde{\cR}_{2,7}} = \ker ( T^*_{(C, \eta)}\widetilde{\cR}_{2,7}  \ra T^*_{(C, \eta)} \cS_i) 
$$
for $i=1,2$.
As in \cite{ds} by taking level structures on the moduli spaces we can assume we are working on fine moduli spaces, which allows to 
identify the tangent space to $\cS_1$ (respectively $\cS_2$) at the $G$-admissible cover $[\tC \ra C]$, where  $C=E /(p \sim q)$   
(respectively $C=E_1 \cup_p E_2$) with the tangent to $\overline{\cM}_2$ at $C$. Thus the conormal bundle $\cN^*_{(C, \eta), \cS_1 /  
\widetilde{\cR}_{2,7}}$  (respectively $\cN^*_{(C, \eta), \cS_2 /  
\widetilde{\cR}_{2,7}}$) can be identified with the conormal bundle $\cN^*_{C,\Delta_0 /  \overline{\cM}_2} 
\subset H^0(\Omega_C \otimes \omega_C)$ (respectively with $\cN^*_{C,\Delta_1 /  \overline{\cM}_2}$ ) where 
$\Delta_0$ denotes the divisor of irreducible nodal curves in $\overline \cM_2$ and $\Delta_1$ the divisor of reducible nodal ones.

Using the fact that $X= \Prym (\tC, C) $ we can identify 
$$
(T_X\cA_D)^* \simeq  \oplus_{i=1}^6 \Omega_i \CC
$$
and $S^2(H^0(\widetilde C, \omega_C)^-)^G$ as in \eqref{eq2} . Then for a covering  $(C,\eta)$ 
the conormal bundles fit in the following commutative diagram 
\begin{equation} \label{diagconormal}
\xymatrix{
0  \ar[r] & \cN^*_{X,\cE/\cB_D}  \ar[r] \ar^{n^*}[d] &   S^2(H^0(\widetilde C, \omega_C)^-)^G
\ar[r] \ar^{\cP^*}[d] &
T^*_X \cE \ar[r] \ar[d] & 0 \\
0 \ar[r] & \cN^*_{(C,\eta), \cS_i/ \widetilde \cR_{2,7}} \ar[r] & H^0(C, \Omega_C \otimes \omega_C) \ar[r] \ar^{H^0(j)}[d] 
& T^*_{(C,\eta)} \cS_i \ar[r] & 0\\
&&H^0(C,\omega_C^2) &&
}
\end{equation}
where  $n^*$ is the conormal map and $i=1,2$. 

\begin{lem}
The kernel of $H^0(j)$ is one-dimensional.
\end{lem}

\begin{proof}
As a map of sheaves the canonical map $j: \Omega_C \otimes \omega_C \ra \omega^2_C$ has 
one-dimensional kernel, namely the one-dimensional torsion sheaf  with support the node of $C$ (see 
\cite[Section IV, 2.3.3]{ds}). On the other hand, the map $H^0(j)$ is the composition of the pullback 
to the normalization with the push forward to $C$. This implies that the kernel of $H^0(j)$ consists
exactly of the sections of the skyscraper sheaf supported at the node and hence is one-dimensional. 
\end{proof}

\begin{prop} \label{p9.5}
For coverings of type {\em(i)}-{\em(iv)} the restricted codifferential map  $n^* : \cN^*_{X, \cE / \cB_D} \ra \cN^*_{(C, \eta), 
\cS_i /  \widetilde{\cR}_{2,7}}$ is surjective, for $i=1,2$.
\end{prop}

\begin{proof}
First notice that from the  ``local-global" exact sequence (see \cite{ba}), $\Ker H^0( j) = \cN^*_{C,\Delta_0 /  \overline{\cM}_2} 
\subset H^0(\Omega_C \otimes \omega_C)$.  So $\Ker H^0(j) \subset \im \cP^*$.  Since $\dim \Ker H^0(j) =1$ and by Proposition \ref{codifDiv}, 
$\dim \Ker (H^0(j) \circ \cP^*) =2$, we have that $\dim \Ker ( \cP^*) =1 $.
By the diagram \eqref{diagconormal} this implies that the kernel of $n^*$ is of dimension $\leq 1$. 
Since $\cN^*_{X,\cE/\cB_D}$ is a vector space of dimension 2 and 
$\cN^*_{(C,\eta), \cS_i/ \widetilde \cR_{2,7}}$ a vector space of dimension 1, it follows that $n^*$
has to be surjective.
\end{proof}


\section{Local degree of $\pr_{2,7}$ over the boundary divisors}


First we compute the local degree of  the Prym map $\widetilde {\pr}_{2,7}$ along the divisor 
$\cS_1$. Since the conormal map of $\pr_{2,7}$ along $\cS_1$ is surjective according to 
Proposition \ref{p9.5}, \cite[I, Lemma 3.2]{ds} implies that the local degree along $\cS_1$ 
is given by the degree of the induced map  $\widetilde \cP: \widetilde S_1 \ra \widetilde \cE$ 
on the exceptional divisor $\widetilde \cS_1$. Now the polarized abelian variety $X(E)$ is uniquely 
determined by the elliptic curve $E$ according to its definition. Hence the curve $\cE$ can be 
identified with the moduli space of elliptic curves, i.e. with the affine line. The exceptional divisor 
$\widetilde \cE$ is then a $\PP^1$-bundle over $\cE$. 
On the other hand, $\cS_1$ is a divisor in $\widetilde \cR_{2,7}$, so $\widetilde \cS_1$ is isomorphic 
to $\cS_1$. 
Clearly  $\widetilde \cP$ maps the fibers $\widetilde {\pr}^{-1}(X(E)) \cap \widetilde \cS_1$ onto 
the fibers $\PP^1$ over the elliptic curves $E$.  
 
Now $\widetilde {\pr}^{-1}(X(E)) \cap \widetilde \cS_1$ consists of coverings of type (i) and one
covering of type  (iv), that we denote by $\cC^{(iv)}_E$. The coverings of type (i) have  as 
base a nodal curve of the 
form $C=E/p\sim q$ and we can assume that $p=0$, thus $\widetilde {\pr}^{-1}(X(E)) \cap \widetilde \cS_1$ 
is parametrized by $E$ itself (the point $q=0$ corresponds to the covering of type (iv)). 
Hence
the induced conormal map on the exceptional divisors $\widetilde{\cP} : \widetilde{\cS}_1\ra   
\widetilde{\cE}$ restricted to 
a the fiber over $X(E)$ is a map $\phi: E \ra \PP^1$.  
Combining everything we conclude that the local degree of the Prym map along $\cS_1$ coincides with the degree
of the induced map $\phi: E \ra \PP^1$.

\begin{prop} \label{degS1}
 The local degree of  the Prym map $\widetilde {\pr}_{2,7}$ along 
$\cS_1  $ is two.
\end{prop}

\begin{proof}
According to what we have written above, it is sufficient to show that the map  $\phi: E \ra \PP^1$ induced by  $\widetilde{\cP} $ is a 
double covering. We use again the identification \eqref{eq2} (respectively its analogue
for coverings of type (iii)).  As in \cite{ds}, let $x,y$ be local coordinates at 0 and $q$ and $dx$, $dy$ the corresponding differentials. 
If $(a,b,c) \in S^2(H^0(\tC, \omega_{\tC})^-)$ are coordinates in the basis of Lemma \ref{basis},  
 then $\PP (\Ker (H^0(j) \circ \cP^*)) \simeq \PP^1$ has coordinates $[a,b]$ and  its dual is identified with 
$\PP(\Ima \widetilde{\cP}_{|\textnormal{fibre}})$.  In order to describe the kernel of $\cP^*$, we look at the multiplication 
on the stalk over the node $p=(q\sim 0)$.
Around $p$ the line bundles $\eta^i$ are trivial, then the element 
$(a, b, c) \in \oplus_{i=1}^{3} (\omega_{C,p} \otimes \omega_{C,p})$ (in coordinates $a,b,c \in \cO_p$ 
for a fixed basis of $\omega_{C,p} \otimes \omega_{C,p}$)  is sent to  $a+b+c \in (\Omega_C \otimes \omega_C)_p$ under
$\widetilde{\cP}$.  Thus  the germ $a+b+c \in \cO_p$ is zero if it is in the kernel of $\cP^*$. In particular, the coefficient 
of $dxdy$ must vanish.  Let  $\alpha=a_0 dx , \beta=b_0 dx,  \gamma=c_0 dx $ and  $\alpha=a_q dy , \beta=b_q dy,  \gamma=c_q dy$
be the local description of the differentials,  then the coefficient of $dxdy$ must satisfy:
\begin{equation} \label{condition}
a_0a_q +b_0b_q + c_0c_q=0.
\end{equation}
Now, by looking at the dual picture, we consider $\PP^1 = \PP(\Ker \cP^*)^* \subset \PP^{2*}$. Let $E$ be embedded in $\PP^{2*}$ by 
the linear system $|3\cdot 0|$.  The coordinate functions $[a,b, c] \in \PP^{2*}$ satisfy condition \eqref{condition} for all $q\in E$. Then 
the points on the fiber over $[a_q , b_q, c_q]$ are points in $E$ over the line passing through the origin $0\in E \subset \PP^{2*}$ and
$q$. Hence the map $E \ra \PP^1$ corresponds to the restriction to $E$ of the projection $\PP^{2*} \ra \PP^1$ from the origin,
which is the double covering $E \ra \PP^1$ determined by the divisor $0 + q$ of $E$ and thus of degree two.
\end{proof}

We turn now our attention to the Prym map on $\cS_2 $. By the surjectivity of  the conormal map of $\pr_{2,7}$ on $\cS_2$ 
(Proposition \ref{p9.5}),  the local degree along $\cS_2$  is computed by the degree of the map  
$\widetilde \cP: \widetilde S_2 \ra \widetilde \cE$ on the divisor $\widetilde \cS_2$, which is a $\PP^1$-bundle over
$\cE$. Given an elliptic curve $E$ the fiber of $\widetilde{\pr}^{-1}(X(E))$ intersected with the  divisor $\cS_2$ consists of coverings of type (ii), 
one covering of type (iii), denoted by $\cC_E^{(iii)}$, and one
covering of type (iv)  $\cC_E^{(iv)}$ which lies in the intersection with the divisor $\cS_1$. 

Recall that the type (ii) coverings  have  base curve  $C=E_1 \cup E$ intersecting at one point that we can assume to be $0$ and 
$E_1$ is an arbitrary elliptic curve. The covering over $C$  is the union of a degree-7 \'etale cyclic covering $F_1$ over $E_1$
and 7 elliptic curves $E_i$ attached to $F_1$ mapping each one of them isomorphically to $E$. 
So  the  type (ii) coverings on the  fiber over $E$ are parametrized by pairs $(E_1, \langle \eta \rangle)$ where $E_1$ is an elliptic curve
and $\langle \eta \rangle \subset E_1$ is a subgroup of order 7. 

It is know that  the parametrization space of the pairs $(E_1, \langle \eta \rangle)$ is the modular curve  $Y_0(7):= \Gamma_0(7) / \mathbb{H}$.  
The natural projection $(E_1, \eta) \mapsto E_1$ defines  a map $\pi_0: Y_0(7) \ra \CC $. Moreover, the curve $Y_0(7)$ admits a compactification 
$X_0(7):= \overline{Y_0(7)}$ such that the map $\pi_0$ extends to a map $\pi: X_0(7) \ra \PP^1$ (see \cite{s1}).
The genus of $X_0(7)$ can be computed by Hurwitz formula 
using the fact that $\pi$ is of degree 8 and it is ramified over the points corresponding to elliptic curves with $j$-invariant 0 and 
$12^3$ (with ramification degree 4 on each fibre) and over $\infty$, where  the inverse image consists of two cusps,
one \'etale and the other of ramification index $7$. The two cusps over $\infty$ represent the coverings $\cC_E^{(iii)}$  and $\cC_E^{(iv)}$
above $X(E)$ (see Remark \ref{polygons}). This gives that $X_0(7)$ is of genus zero. 
Thus, we can identify ${\widetilde \cS_2} \cap  \widetilde{\pr}^{-1}(X(E))$ with  $X_0(7) \simeq \PP^1$.

Then, since $\widetilde{\cS}_2 \simeq \cS_2 $,  the restriction of the conormal map 
$\widetilde \cP$ to a fiber over the point $[E] \in \cE$ is a map $\psi: \PP^1 \ra \PP^1$.

\begin{prop}
 The map $\pi$ coincides with the map $\psi: \PP^1 \ra \PP^1$ of the fibers of $\widetilde{\cS}_2  \ra \widetilde{\cE}$  over a point $[E] \in \cE$. 
\end{prop}

\begin{proof}

 Let $o$, respectively $o'$ be the zero element of $E$ respectively $E_1$ with local coordinates $x$ respectively $y$. Set
 $\alpha=a_o dx , \beta=b_o dx,  \gamma=c_o dx $ and  $\alpha=a_{o'} dy , \beta=b_{o'} dy,  \gamma=c_{o'} dy$ the local
 description of elements of  $(\omega_{C,p} \otimes \omega_{C,p})$ around the node $p= (o \sim o')$.  As in the proof of Proposition  \ref{degS1}, 
we have that for an element $(a,b,c)$ in the kernel of $\cP^*$ the coefficient of $dxdy$ must vanish, i.e.  it satisfies
\begin{equation} \label{condition2}
a_oa_{o'} +b_ob_{o'} + c_oc_{o'}=0.
\end{equation}
in a neighborhood of the node.
Considering the dual map one sees that the fiber of $\widetilde \cP$ over a point $[a_o, b_o, c_o ] \in \PP(\Ker \cP^*)^* \subset \PP^2$
with $c_o = -a_o-b_o$,
corresponds to the pairs $(E_1, \langle \eta \rangle)$ such that the local functions $a,b,c \in \cO_p$ take the values  $[a_{o'}, b_{o'}, 
c_{o'}]$ around the $o' \in E_1$ with $c_{o'} = -a_{o'}-b_{o'}$ and they verify \eqref{condition2}.  This determines completely the triple 
$[a_{o'}, b_{o'}, c_{o'}]$ and it depends only on the values at the node $o' \in E_1$ of the base curve. 
Note that $a,b,c$ are elements of the local ring $\cO_p$ which determines the curve $E_1$ uniquely. In fact, its 
quotient ring is the direct product of the function fields of $E_1$ and $E$ which in turn 
determines the curves. Therefore the map $\psi$ can be identified with the projection $(E_1, \langle \eta \rangle) \mapsto E_1$.
\end{proof}

As an immediate consequence we have

\begin{cor}\label{degS2}
 The local degree of  the Prym map $\widetilde {\pr}_{2,7}$ along the divisor 
$\cS_2$ is 8.
\end{cor}

Using Proposition \ref{degS1} and Corollary \ref{degS2} we conclude that the degree of the Prym map $\widetilde{\pr}_{2,7}$
is 10, which finishes the proof of Theorem \ref{main-theorem}

\begin{rem} \label{polygons}
The moduli interpretation of $X_0(7) \setminus Y_0(7)$ is given by the {\it N\'eron polygons}:  one of the cusps represents a {\it 1-gon},
that is a nodal cubic curve, corresponding to the covering $\cC_E^{(iii)}$  and the other represents a {\it 7-gon}, that is
7 copies of $\PP^1$ with the point 0 of one attached to the point $\infty$ of other in a closed chain, which corresponds to the covering 
$\cC_E^{(iv)}$ (see \cite[IV.\S 8]{s2}).
\end{rem}

\end{document}